\documentclass[12pt]{amsart}
\usepackage{amssymb}
\usepackage{enumerate}
\usepackage{amsthm}
\usepackage{tikz-cd}
\usepackage{amsmath}
\makeatletter
\@namedef{subjclassname@2010}{%
  \textup{2010} Mathematics Subject Classification}
\makeatother
\newtheorem{thm}{Theorem}[section]
\newtheorem{lem}[thm]{Lemma}
\newtheorem{prop}[thm]{Proposition}
\newtheorem{defi}[thm]{Definition}

\newtheorem{corl}[thm]{Corollary}
\newtheorem{xrem}{Remark}

\begin{document}
\baselineskip=17pt

\subjclass[2010]{Primary 14J26, 14J60; Secondary 14H60, 14J10}
\keywords{Semistable; Higgs bundle; Moduli spaces}
\author{Snehajit Misra}

\address{The Institute of Mathematical Sciences, HBNI, CIT Campus, Taramani, Chennai 600113, India.}
\email[Snehajit Misra]{snehajitm@imsc.res.in}

\begin{abstract}
Let $\pi : X = \mathbb{P}_C(E) \longrightarrow C$ be a ruled surface over an algebraically closed field $k$ of characteristic 0, with a fixed polarization $L$ on $X$. In this paper, we show that pullback of a (semi)stable Higgs bundle on $C$ under $\pi$ is a $L$-(semi)stable Higgs bundle. Conversely, if $(V,\theta)$ is a $L$-(semi)stable Higgs bundle on $X$ with $c_1(V)= \pi^*(\bf d \rm)$ for some divisor \bf d \rm  of degree $d$ on $C$ and $c_2(V)=0$, then there exists a (semi)stable Higgs bundle $(W,\psi)$ of degree $d$ on $C$ whose pullback under $\pi$ is isomorphic to $(V,\theta)$. As a consequence, we get an isomorphism between the corresponding moduli spaces of (semi)stable Higgs bundles. We also show the existence of non-trivial stable Higgs bundle on $X$ whenever $g(C)\geq 2$ and the base field is $\mathbb{C}$.
\end{abstract}

\title{ Stable Higgs Bundles on ruled surfaces}
\maketitle

\vskip 4mm

\section{Introduction}
A Higgs bundle on an algebraic variety $X$ is a pair ($V,\theta$) consisting of a vector bundle $V$ over $X$ together with a Higgs field $\theta : V \longrightarrow V \otimes \Omega^1_X$ such that 
$\theta \wedge \theta = 0$. Higgs bundle comes with a natural stability condition (see Definition \ref{defn1} for stability), which allows one to study the moduli spaces of stable Higgs bundles on $X$. Higgs bundles on Riemann surfaces 
were first introduced by Nigel Hitchin in 1987 and  subsequently, Simpson extended this notion on higher dimensional varieties. Since then, these objects have been studied by many authors, but very little is known about stability of Higgs bundles 
on ruled surfaces.

  Let $\pi : X = \mathbb{P}_C(E) \longrightarrow C$ be a ruled surface over an algebraically closed field $k$ of characteristic 0, where $C$ is a smooth irreducible projective curve of genus $g(C) \geq 0$. We fix a polarization $L$ on $X$.
  In this paper, our main results are the following, 
  
\subsection*{\bf Theorem \bf\ref{thm3.3}}
\it Let  $\pi:X \longrightarrow C $  be a ruled surface with a fixed polarization $L$ on $X$. Let $\mathcal{E} = (V,\theta)$ be a semistable Higgs bundle of rank $r$ on $C$. Then, the pullback $\pi^*(\mathcal{E})=(\pi^*(V),d\pi(\theta))$ 
is $L$-semistable Higgs bundle on $X$.

\subsection*{\bf Theorem \bf\ref{thm3.5}}
\it Let $L$ be a fixed polarization on a ruled surface $\pi: X\longrightarrow C $. Let $\mathcal{E}=(V,\theta)$ be a $L$-semistable Higgs bundle of rank $r$ on $X$ with $c_1(V)=\pi^*(\bf d\it)$, for some divisor \bf d \it of degree $d$ on $C$,
then $c_2(V) \geq 0$ and $c_2(V)= 0$ iff there exists a semistable Higgs bundle $\mathcal{W}=(W,\psi)$ on $C$ such that $\pi^*(\mathcal{W})=(\pi^*(W),d\pi(\psi)) \cong \mathcal{E}$ on $X$.
 
\subsection*{\bf Theorem \bf\ref{thm4.1}} 
  Let $L$ be a fixed polarization on a ruled surface $\pi: X\longrightarrow C $. Then, for any stable Higgs bundle $\mathcal{W} = (W,\psi)$ on $C$, the pullback Higgs bundle $\pi^*(\mathcal{W})$ is $L$-stable Higgs bundle on $X$. 
Conversely, if $\mathcal{V} = (V,\theta)$ is a $L$-stable Higgs bundle on $X$ with $c_1(V) = \pi^*(\bf d \it)$ for some divisor \bf d \it of degree $d$ on $C$ and $c_2(V)= 0$, then $\mathcal{V} \cong \pi^*(\mathcal{W})$ for some stable Higgs
bundle $\mathcal{W}=(W,\psi)$ on $C$.\rm

\vskip 1.8mm
Note that, in particular, taking the Higgs field $\theta=0,\psi=0$ in Theorem \ref{thm4.1}, we recover a well known result for ordinary vector bundles of rank 2 on ruled surfaces which has been proved by Fumio Takemoto (See in \cite{12} Proposition 3.4 
and Proposition 3.6) as well as by Marian Aprodu, Vasile Br\^{i}nz\v{a}nescu independently (see in \cite{1} corollary 3 ). Although, our approach in this paper is different from that of \cite{12} and \cite{1}.

 According to Simpson ( See \cite{9}, \cite{10}), the moduli space of $S$-equivalence classes of semistable rank $n$ Higgs bundles with vanishing Chern classes on any complex projective variety $X$ 
 can be identified with the space of isomorphism classes of representations of $\pi_1(X,*)$ in $GL(n,\mathbb{C})$. For a ruled surface, $ \pi : X = \mathbb{P}_C(E) \longrightarrow C$, there is an isomorphism of fundamental groups,
 $\pi_1(X,*) \cong \pi_1(C,*)$. Hence, for a ruled surface $X$ over a curve $C$, ( when the base field is $\mathbb{C}$), we have a natural algebraic isomorphism of the corresponding moduli of semistable 
 Higgs bundles on $X$ and $C$ respectively with vanishing Chern classes. In this paper, we prove a similar algebraic isomorphism between the corresponding moduli spaces of Higgs bundles when the Chern classes are not necessarily vanishing and 
 the base field is any algebraically closed field $k$ of characteristic 0. More precisely, we prove that,

\subsection*{\bf Theorem \bf\ref{thm6}}
  \it The moduli spaces $\mathcal{M}_X^{Higgs}(r,\pi^*(\bf d \rm),0)$ and $\mathcal{M}^{Higgs}_C(r,d)$ are isomorphic as algebraic varieties, where  $\mathcal{M}^{Higgs}_C(r,d)$ denotes the moduli space of S-equivalence classes of semistable Higgs bundles of rank $r$ and degree $d$ on $C$ and $\mathcal{M}_X^{Higgs}(r,\pi^*(\bf d \rm),0)$ denotes the moduli space of S-equivalence 
 classes of $L$-semistable Higgs bundles of rank $r$ on $X$, having vanishing second chern class and first chern class of the form $\pi^*(\bf d \rm)$ for some divisor \bf d \it of degree $d$ on $C$.\rm
 
 \vspace{4mm}

In \cite{13} , similar kind of questions are discussed for a relatively minimal non-isotrivial elliptic surfaces $\pi : X \longrightarrow C$ over the field of complex numbers, whenever $g(C) \geq 2$.
\section{Preliminaries}
 All the algebraic varieties are assumed to be irreducible and defined over an algebraically closed field $k$ of characteristic 0 unless otherwise specified. In this section, we recall the definition and basic properties of Higgs bundle. 
 We refer the reader to \cite{9} and \cite{5} for more details.
 
 \subsection{Definitions and Conventions}
 Let $X$ be a smooth projective variety of dimension $s$. By a polarization on $X$, we mean, a ray $\mathbb{R}_{>0}\cdot L$ where $L$ is in the ample cone inside the real  N\'{e}ron-Severi group Num$(X)_{\mathbb{R}}$.  
 If $F$ is a coherent sheaf on $X$, then rank of $F$ is defined as the rank of the $\mathcal{O}_\xi$ -vector space $F_\xi$, where $\xi$ is the unique generic point of $X$. Note that, $F$ is a torsion sheaf iff rank of $F$ is 0.
 Let $L$ be a fixed polarization on $X$ and $F$ be a torsion-free coherent sheaf of rank $r$ on $X$. The slope of $F$ with respect to $L$ is defined by
\begin{align*}
\mu_L (F) :=\frac{c_1(F).L^{s - 1}}{r}.
\end{align*} 
\begin{defi}
\rm A torsion free sheaf $F$ is said to be slope $L$-semistable $($resp. slope $L$-stable$)$ if for any coherent subsheaf $G$ of $F$ with 0 $<$ rank$(G) < $rank$(F)$, one has $\mu_L(G)\leq \mu_L(F) ($resp. $\mu_L(G) < \mu_L(F))$.
\end{defi}
\begin{defi}
\rm A Higgs sheaf $\mathcal{E}$ on $X$ is a pair $(E,\theta$), where $E$ is a coherent sheaf on $X$ and $\theta: E \longrightarrow E \otimes \Omega^1_X$ is a morphism of $\mathcal{O}_X$-module such that $\theta \wedge \theta$ = 0, 
where $\Omega^1_X$ is the cotangent sheaf to $X$ and $\theta \wedge \theta$ is the composition map
\begin{align*}
E \longrightarrow E\otimes \Omega^1_X \longrightarrow E \otimes \Omega^1_X \otimes \Omega^1_X \longrightarrow E \otimes \Omega^2_X.
\end{align*}
\end{defi}
   $\theta$ is called the Higgs-field of $\mathcal{E}$. A Higgs bundle is a Higgs sheaf $\mathcal{V}=(V,\theta)$ such that $V$ is a locally-free $\mathcal{O}_X$-module. If $\mathcal{E}=(E,\phi)$ and $\mathcal{G}=(G,\psi)$ are Higgs sheaves, 
   a morphism $f:(E,\phi)\longrightarrow(G,\psi)$ is a morphism of $\mathcal{O}_X$-modules $f:E\longrightarrow G$ such that the following diagram commutes.
\begin{center}
 \begin{tikzcd}
E \arrow[r, "f"] \arrow[d, "\phi"]
& G \arrow[d, "\psi" ] \\
E \otimes\Omega^1_X \arrow[r, "f\otimes id" ]
& G\otimes\Omega^1_X
\end{tikzcd}
\end{center}

$\mathcal{E}$ and $\mathcal{G}$ are said to isomorphic if there is a morphism $f : \mathcal{E} \longrightarrow \mathcal{G}$ such that $f$ as an $\mathcal{O}_X$-module map is an isomorphism.

\begin{defi}\label{defn1}
 \rm A  Higgs sheaf  $\mathcal{E} = (E,\theta)$ is said to be $L$-semistable $($resp. $L$-stable$)$ if $E$ is torsion-free and for every $\theta$-invariant  subsheaf $G$ of $E$ $($i.e. $\theta(G) \subset G\otimes \Omega^1_X)$ with  
 $0 < $ rank$(G)<$ rank$(E)$, one has $\mu_L(G)\leq \mu_L(E)$ $($resp. $\mu_L(G) < \mu_L(E))$.
\end{defi}  
\begin{xrem}\label{remark1}
 \rm Moreover, when $X$ is a smooth projective curve or a surface, in the definition of semistability(resp. stability) for a Higgs bundle $\mathcal{V}=(V,\theta)$, it is enough to consider $\theta$-invariant subbundles $G$ of $V$ with $0 < $ 
 rank$(G)<$ rank$(V)$ for which the quotient $V/G$ is torsion-free. It is clear from the definition that, for a $L$-semistable (resp. stable) Higgs bundle $\mathcal{E} = (V,\theta)$ with zero Higgs field (i.e. $\theta$=0), the underlying vector 
 bundle $V$ itself slope $L$-semistable (resp. stable). Also, a slope $L$-semistable(resp. stable) vector bundle on $X$ is Higgs semistable(resp. stable) with respect to any Higgs field $\theta$ defined on it. 
 If $X$ is a smooth projective curve, then for a torsion-free sheaf $F$ of rank $r$, $\mu_L(F)$ is independent of the choice of the polarization $L$. Hence, whenever (semi)stability of bundles will be talked on a curve, 
 the polarization will not be mentioned.
\end{xrem} 
 \subsection*{}
 For a smooth map $\phi : X \longrightarrow Y $ between two smooth projective varieties $X$ and $Y$, and a Higgs bundle $\mathcal{E}=(V,\theta)$ on $Y$, its pullback $\phi^*(\mathcal{E})$ under $\phi$ is defined as the Higgs bundle
 ($\phi^*(V),d\phi(\theta))$, where $d\phi(\theta)$ is the composition map
 \begin{align*}
  \phi^*(V)\longrightarrow \phi^*(V) \otimes \phi^*(\Omega^1_Y) \longrightarrow \phi^*(V)\otimes \Omega^1_X
 \end{align*}
If $\phi : X \longrightarrow Y $ is a finite separable morphism of smooth projective curves, then a Higgs bundle $\mathcal{E}$ is semistable on Y  iff $\phi^*(\mathcal{E})$ is semistable on $X$. See Lemma 3.3 in \cite{2} for the proof.
 \subsection*{}
 For a $L$-semistable Higgs bundle $(V,\theta)$ on $X$, there is a filtration of $\theta$-invariant subbundle 
 \begin{align*}
  0 \subset V_0 \subset V_1 \subset V_2 \subset ....\subset V_{k-1} \subset V_k = V
 \end{align*}
called the Jordan-H\"{o}lder filtration, where for every $i=0,1,...,k$, $\mu_L(V_i/V_{i-1}) = \mu_L(V)$ and the induced Higgs sheaf $(V_i/V_{i-1},\overline{\theta}\vert_{V_{i-1}})$ is $L$-stable. This filtration is not unique, but the graded sheaf
$ \mathcal{G}r(V,\theta) = \bigoplus^k_{i=0} (V_i/V_{i-1},\overline{\theta}\vert_{V_{i-1}})$ is unique upto isomorphism. Two $L$-semistable Higgs bundle are said to be S-equivalent if their corresponding graded sheaves are isomorphic. (See \cite{9})

\subsection{Bogomolov's Inequality}    
 Let $F$ be a coherent sheaf on $X$ with Chern classes $c_i$ and rank $r$. The discriminant of $F$ by definition is the characteristic class 
\begin{align*}
 \triangle(F) = 2rc_2 - (r-1)c_1^2
\end{align*}
   Let $X$ be a smooth projective surface and $\mathcal{E} = (V,\theta)$ be a semistable Higgs bundle with respect to a fixed polarization $L$ on $X$. Then, Bogomolov Inequality says that $\triangle(V) \geq 0$ ( See Proposition 3.4 in \cite{8}). 
   Recently, it has been proved that the Bogomolov's inequality also holds true in characteristic 0 for semistable Higgs sheaf (See Theorem 7 in \cite{7}).

 \subsection{Ruled Surfaces}
 Let $ C $  be a smooth projective algebraic curve of genus \it g\rm over an algebraically closed field $k$ with char$(k)$ = 0.   A \it geometrically ruled surface \rm  or simply \it ruled surface\rm, is a surface $X$, together with a surjective morphism
 $\pi: X \longrightarrow C$ such that the fiber $X_y$ is isomorphic to $\mathbb{P}^1_{k}$ for every closed point $y \in C$, and such that $\pi$ admits a section(i.e. a morphism $ \sigma : C \longrightarrow X$ such that $\pi \circ \sigma$ = id$_C$). 
 Equivalently, $X \cong \mathbb{P}_C(E)$ over $C$ for some rank 2 vector bundle $E$ on $C$. Moreover, if $E_1$ and $E_2$ are two vector bundle of rank 2 on $C$, then $ \mathbb{P}_C(E_1) $ and $\mathbb{P}_C(E_2)$ are isomorphic as ruled surfaces over 
 $C$ iff there is a line bundle $N$ on $C$ such that $ E_1 \cong E_2 \otimes N $. See \cite{5} (chapter V, section 2) for more details.
 
            Let $\sigma$ be a section and $f$ be a fiber of the ruling $\pi: X \longrightarrow C$. Then, 
            
            1) Pic$X\simeq \mathbb{Z}\cdot\sigma \oplus \pi^*($Pic$C$) and 
            
            2) Num$X\simeq \mathbb{Z}\cdot\sigma \oplus \mathbb{Z}\cdot f$ satisfying $\sigma\cdot f=1 , f^2=0$

 \section{Semistability under pullback}
 The next lemma is due to  \cite{4} or \cite{11} which we will use repeatedly to prove our main results.

\begin{lem}\label{lemma3.1}
\rm Let $F$ be a torsion free sheaf of rank $r$ on a ruled surface $\pi : X \longrightarrow C $  and $F\vert_f \cong \mathcal{O}_f^{\oplus r}$ for a generic fiber $f$ of the map $\pi$. Then, $c_2(F)\geq 0$ and $c_2(F)= 0$ iff $F \cong \pi^*(V)$ for 
some vector bundle $V$ on $C$.
\begin{proof}
See Lemma 2.2 in \cite{11}.
\end{proof}
\end{lem}

\begin{prop}\label{prop3.2}
\rm If $\mathcal{E} = (V,\theta)$ is $L_1$-semitable Higgs bundle on a smooth algebraic surface $X$ with a fixed polarization $L_1$ on $X$ and $\triangle(V)=0 $, then the semistability of the Higgs bundle $\mathcal{E}=(V,\theta)$ is independent 
of the polarization chosen.
\end{prop}
\begin{proof}
 Suppose there is a polarization $L_2$ such that $\mathcal{E} = (V,\theta)$ is not $L_2$-semistable Higgs bundle. Then, there exist a saturated $\theta$-invariant subsheaf $V_0 \subseteq V$ with $\mu_{L_2}(V_0) > \mu_{L_2}(V)$. Let $V'$ is any $\theta$-
 invariant saturated subsheaf with this property. Define, 
 \begin{align*}
  r(V') := \frac{\mu_{L_1}(V) - \mu_{L_1}(V')}{\mu_{L_2}(V') - \mu_{L_2}(V)}.
 \end{align*}
Then, $\mu_{L_1+r(V')L_2}(V') = \mu_{L_1+r(V')L_2}(V)$. We note that $L_0 := L_1 + r(V_0)L_2$ is a polarization on $X$. If $r(V') < r(V_0)$, then $\mu_{L_0}(V') > \mu_{L_0}(V)$. By Grothendieck's Lemma ( See Lemma in 1.7.9 in \cite{6} ) , the family of saturated subsheaves
$V'$ with $\mu_{L_0}(V') > \mu_{L_0}(V)$ is bounded. So, there are only finitely many numbers $r(V')$ which are smaller than $r(V)$. We can further choose $V_0$ in such a way that $r(V_0)$ is minimal. Then, $V$ and $V_0$ are $L_0$-Higgs semistable with 
$\mu_{L_0}(V_0) = \mu_{L_0}(V)$. So, we have an exact sequence of torsion free sheaves 
\begin{align}\label{seq3.2.1}
0 \longrightarrow  V_0 \longrightarrow V \longrightarrow V_1 \longrightarrow 0
\end{align}
 with $\mu_{L_0}(V_0) = \mu_{L_0}(V) = \mu_{L_0}(V_1)$.
 
 Let $\overline{\theta}$ : $V_1 \longrightarrow V_1 \otimes \Omega^1_X$ be the induced map. Our claim is that $(V_1,\overline{\theta})$ are $L_0$-semistable Higgs sheaves.  
Now, let $\overline{W_1}$ be a $\overline{\theta}$-invariant subsheaf of $V_1$. Then, we have an exact sequence of torsion-free sheaves
\begin{align*}
 0\longrightarrow V_0 \longrightarrow W_1 \longrightarrow \overline{W_1}\longrightarrow 0
\end{align*}
where $W_1$ is $\theta$-invariant subsheaf of $V$ containing $V_0$. We have, $\mu_{L_0}(W_1) \leq \mu_{L_0}(V) = \mu_{L_0}(V_0)$. We also have, by Lemma 2 (Chapter 4 in \cite{3})
\begin{align*}
min(\mu_{L_0}(V_0),\mu_{L_0}(\overline{W}_1)) \leq \mu_{L_0}(W_1) \leq max(\mu_{L_0}(V_0),\mu_{L_0}(\overline{W}_1))
\end{align*}
Hence, $\mu_{L_0}(\overline{W}_1) \leq \mu_{L_0}(W_1) \leq \mu_{L_0}(V) = \mu_{L_0}(V_1)$. Hence,$(V_1,\overline{\theta})$ is $L_0$-semistable Higgs sheaf. Therefore, by Bogomolov's Inequality, $\triangle({V_0}) \geq 0, \triangle({V_1}) \geq 0$.

 We denote $\xi \equiv (r.c_1(V_0)-r_0.C_1(V)) \in $ Num$(X)_\mathbb{R}$ where $r$ and $r_0$ denotes the ranks of $V$ and $V_0$ respectively. Hence, $\xi.L_0 = 0$ and $\xi.L_2 >  0$. So, by Hodge Index Theorem, $\xi^2 < 0$. 
 On the other hand, from the exact sequence (\ref{seq3.2.1}) we have,
 \begin{align*}
  0=\triangle(V)=\frac{r}{r_1}\triangle(V_0)+\frac{r}{r-r_1}\triangle(V_1)-\frac{\xi^2}{r_1(r-r_1)}
 \end{align*}
Since $\triangle(V_0) \geq 0  $ and $\triangle(V_1) \geq 0$, we have, $\xi^2 \geq 0$ which is a contradiction. Hence, our result is proved.
\end{proof}

\begin{xrem}\label{remark2}
\rm  A similar argument as in Proposition \ref{prop3.2} will imply that if $\mathcal{E} = (V,\theta)$ is $L$-stable Higgs bundle on a smooth algebraic surface $X$ with a fixed polarization $L$ on $X$ and $\triangle(V)=0 $, then the stability of the Higgs 
bundle $\mathcal{E}=(V,\theta)$ is independent of the polarization chosen. 
\end{xrem}

\begin{thm}\label{thm3.3}
\rm Let  $\pi:X \longrightarrow C $  be a ruled surface with a fixed polarization $L$ on $X$. Let $\mathcal{E} = (V,\theta)$ be a semistable Higgs bundle of rank $r$ on $C$. Then, the pullback $\pi^*(\mathcal{E})=(\pi^*(V),d\pi(\theta))$ 
is $L$-semistable Higgs bundle on $X$. 
\begin{proof}
Let $H$ be a very ample line bundle on $X$. By Bertini's Theorem, there exist a smooth projective curve $B$ in the linear system $\mid H \mid$ . Let us consider the induced map between smooth projective curves 
\begin{align*}
 \pi_B:B\hookrightarrow X \longrightarrow C . 
\end{align*}

   Since $B.f=H.f > 0$ , $B$ is not contained in any fiber. Hence, $\pi_B$ is a finite separable morphism between two smooth projective curves and by Lemma 3.3 in \cite{2}, $\pi^*_B (\mathcal{E})$ is a semistable Higgs bundle on $B$. 
   Now, suppose  $\pi^*(\mathcal{E})$ is not $H$-semistable Higgs bundle, then there exist $d\pi(\theta)$-invariant subbundle $W$ of $\pi^*(V)$ such that $\mu_H(W) > \mu_H(\pi^*(V))$. Hence, we have
\begin{align*}
\mu(W\vert_B) > \mu(\pi^*(V)\vert_B)
\end{align*}
But, $W\vert_B$ is a $d\pi_B(\theta) $-invariant subbundle of $\pi_B^*(V)=\pi^*(V)\vert_B$. Thus, $\pi_B^*(\mathcal{E}) = (\pi_B^*(V),d\pi_B(\theta))$ is not a semistable Higgs bundle on 
 $B$, which is a contradiction. Therefore, $\pi^*(\mathcal{E})=(\pi^*(V),d\pi(\theta))$ is a $H$-semistable Higgs bundle on $X$. Now the discriminant of $\pi^*(V)$ being  0, by Proposition \ref{prop3.2}, $\pi^*(\mathcal{E})=(\pi^*(V),d\pi(\theta))$ is a
 $L$-semistable Higgs bundle on $X$ for any polarization $L$ on $X$.  
\end{proof}
\end{thm}

\begin{prop}\label{prop3.4}
\rm Let $\pi: X = \mathbb{P}_C(E)\longrightarrow C $ be a ruled surface for some rank 2 vector bundle $E$ on $C$ . Then, the natural map $\Omega^1_C \xrightarrow {\eta} \pi_*(\Omega^1_X)$ is an isomorphism.
 \begin{proof}
  Consider the exact sequence 
  \begin{align}\label{seq3.4.1}
   0 \longrightarrow \pi^*(\Omega^1_C) \longrightarrow \Omega^1_X \longrightarrow \Omega^1_{X/C} \longrightarrow 0
  \end{align}
  Applying $\pi_*$ to the exact sequence (\ref{seq3.4.1}), we get the following long exact sequence,
  \begin{align}\label{seq3.4.2}
   0 \longrightarrow \Omega^1_C \longrightarrow \pi_*(\Omega^1_X) \longrightarrow \pi_*(\Omega^1_{X/C}) \longrightarrow \Omega^1_C \otimes R^1\pi_*(\mathcal{O}_X) \longrightarrow ...
  \end{align}
We also have 
\begin{align}\label{seq3.4.3}
 0\longrightarrow \Omega^1_{X/C} \longrightarrow (\pi^*(E))\otimes \mathcal{O}_{\mathbb{P}(E)}(-1) \longrightarrow \mathcal{O}_{\mathbb{P}(E)} \longrightarrow 0
\end{align}
Since $\pi$ is a smooth map of relative dimension 1 between two nonsingular varieties , by Proposition 10.4 in \cite{5}(chapter III, Page 270), $\Omega^1_{X/C}$ is a locally free sheaf of rank 1 on $X$. Applying $\pi_*$ to exact sequence(\ref{seq3.4.3}), 
we get
\begin{align}\label{seq3.4.4}
 0\longrightarrow \pi_*(\Omega^1_{X/C})\longrightarrow \pi_*((\pi^*(E))\otimes \mathcal{O}_{\mathbb{P}(E)}(-1)) \longrightarrow .. 
\end{align}
By Projection formula, we have $ \pi_*((\pi^*(E))\otimes \mathcal{O}_{\mathbb{P}(E)}(-1)) = E \otimes \pi_*(\mathcal{O}_{\mathbb{P}(E)}(-1)) = 0.$ Therefore, from exact sequence (\ref{seq3.4.4}), we get, $\pi_*(\Omega^1_{X/C}) = 0 $ and 
hence from exact sequence (\ref{seq3.4.2}), we have, the natural map $ \Omega^1_C \xrightarrow{\eta} \pi_*(\Omega^1_X)$ is an isomorphism.
\end{proof}
\end{prop}

\begin{thm}\label{thm3.5}
 \it Let $L$ be a fixed polarization on a ruled surface $\pi: X\longrightarrow C $. Let $\mathcal{E}=(V,\theta)$ be a $L$-semistable Higgs bundle of rank $r$ on $X$ with $c_1(V)=\pi^*(\bf d\it)$, for some divisor \bf d \it of degree $d$ on $C$. Then,
 $c_2(V) \geq 0$ and $c_2(V) = 0$ iff there exists a semistable Higgs bundle $\mathcal{W}=(W,\psi)$ on $C$ such that $\pi^*(\mathcal{W})=(\pi^*(W),d\pi(\psi)) \cong \mathcal{E}$ on $X$.\rm
\begin{proof}
 By Bogomolov's Inequality, $2rc_2(V) \geq (r - 1)c_1^2(V)$ = 0. Hence, $c_2(V) \geq 0$. 
 
 If $c_2(V)$ = 0, then $\triangle(V)=0$. Our claim is that, in this case, for a generic fiber $f$, $V\vert_f$ is slope semistable vector bundle on $f$ and hence,
 $V\vert_f \cong \mathcal{O}_f^{\oplus r}$ ( as deg($V\vert_f$) = $c_1(V).f$ = 0). If not, then, $V\vert_f = \mathcal{O}_f(a_1) \oplus \mathcal{O}_f(a_2) \oplus \cdot \cdot\cdot \cdot\oplus\mathcal{O}_f(a_r)$ for some integers $a_1, a_2,..,a_r$ 
 such that not all the $a_j's$ are zero . Without loss of generality we assume that $a_1 \geq a_2 \geq a_3 \geq \cdot \cdot \cdot \geq a_r$. Further one can assume that $a_1 > 0$ as deg($V\vert_f$) = $\sum a_j = 0$. Let
 $a_1=a_2= \cdot\cdot \cdot=a_i > a_{i+1}$ for some $1 \leq i <r$. Consider 
 $W_f = \mathcal{O}_f(a_1) \oplus \cdot \cdot \cdot\oplus\mathcal{O}_f(a_i)$. Then, $W_f$ is slope semistable and deg($W_f) > 0$. Consider the exact sequence
\begin{align*}
 0\longrightarrow\pi^*(\Omega^1_C)\longrightarrow\Omega^1_X\longrightarrow\Omega^1_{X\vert C}\longrightarrow0
\end{align*}
Restricting the above exact sequence to a generic fibre $f$, we get 
\begin{align*}
 0\longrightarrow\mathcal{O}_f\longrightarrow\Omega^1_X\vert_f\longrightarrow\Omega^1_{X\vert C}\vert_f\longrightarrow0
\end{align*}

By Corollary 2.11 (in  Chapter 5 of \cite{5}), we have, deg($\Omega^1_{X}\vert_{f}$) = - 2 and hence, $\Omega^1_{X}\vert_f = \mathcal{O}_f\oplus\mathcal{O}_f(-2)$ so that deg($W_f) > $ deg($\mathcal{O}_f(a_l))$ and
 deg($W_f) > $ deg($\mathcal{O}_f(a_l-2)$) for  $(i+1) \leq l \leq r$. As $W_f$ is slope semistable, this implies that there does not exists any non-zero map from 
$W_f$ to $\mathcal{O}_f(a_l)\otimes  \Omega^1_X\vert_{f}$ for $(i+1) \leq l \leq r$.

Now, $V\vert_{f} \otimes \Omega^1_X\vert_{f}$ = $\{ W_f \oplus \mathcal{O}_f(a_{i+1}) \oplus \cdot\cdot\cdot\cdot\cdot \oplus \mathcal{O}_f(a_r) \} \otimes \Omega^1_X\vert_{f}$

= $ (W_f \otimes \Omega^1_X\vert_{f}) \oplus ( \mathcal{O}_f(a_{i+1}) \otimes \Omega^1_X\vert_{f}) \oplus \cdot\cdot\cdot\cdot \oplus (\mathcal{O}_f(a_r) \otimes \Omega^1_X\vert_{f}).$

Hence, $ \theta\vert_{f} : W_f \longrightarrow W_f \otimes \Omega^1_X\vert_{f}$. We extend $W_f$ to a $\theta$-invariant subbundle $W \hookrightarrow V$ 
such that the quotient is also torsion-free. Since $\theta\vert_f$ preserves $W_f$, $W$ is also preserved by $\theta$. Note that, $c_1(W)\cdot f = $deg$(W_f) > 0$.  Hence, for a large $m\gg0$, $\mu_{L+mf}(W) > \mu_{L+mf}(V)$. Since $\triangle(V) = 0$,
this contradicts that $(V,\theta)$ is $(L+mf)$-semistable Higgs bundle. This proves our claim.

Therefore, $V\vert_f \cong \mathcal{O}_f^{\oplus r}$ for generic fiber $f$ and also $c_2(V) = 0$. By Lemma \ref{lemma3.1}, $V \cong \pi^*(W)$ for some vector bundle $W$ on $C$. Note that, by projection formula, we have 

$H^0(X,End(V)\otimes \pi^*(\Omega^1_C)) \cong H^0(X,End(V)\otimes \Omega^1_X)\cong H^0(C,End(W)\otimes \Omega^1_C)$. 

Hence, the natural inclusion map $H^0(X,End(V)\otimes \pi^*(\Omega^1_C)) \hookrightarrow H^0(X,End(V)\otimes \Omega^1_X)$ is also surjective 
i.e. every Higgs-field on $V$ factors through $V\otimes\pi^*(\Omega^1_C).$
Now consider the Higgs-field $\psi$  defined as follows (using Proposition \ref{prop3.4} and projection formula)
  \begin{align*}
   \psi := \pi_*(\theta) : \pi_*(\pi^*(W)) \cong  W \longrightarrow \pi_*(\pi^*(W)\otimes \Omega^1_X) \cong W\otimes \Omega^1_C
  \end{align*}
 Since $C$ is a curve, The condition $\psi \wedge \psi = 0 $ is automatically satisfied. Hence $\mathcal{W} := (W,\psi)$ is a well-defined Higgs bundle on $C$.
 Now, consider the following commutative diagram.
 \begin{center}
 \begin{tikzcd}
V \arrow[r, "\theta"] \arrow[d, "\cong"]
& V\otimes\pi^*(\Omega^1_C) \arrow[r, "id\otimes \eta" ]\arrow[d,"\cong"]
& V\otimes\Omega^1_X \arrow[d,"\cong"]\\
\pi^*(W) \arrow[r, "\pi^*(\psi)" ]
& \pi^*(W)\otimes\pi^*(\Omega^1_C) \arrow[r,"id\otimes\eta"]
& \pi^*(W)\otimes\Omega^1_X
\end{tikzcd}
\end{center}
From the above commutative diagram, we have, $\pi^*(\mathcal{W}) \cong \mathcal{E}$.
 
   Our claim is that $(W,\psi)$ is Higgs semistable on $C$. If not, then, there is a $\psi$-invariant subbundle, say, $W_1$ of $W$ such that $\mu(W_1) > \mu(W)$. This implies 
   \begin{align*}
    \mu_L(\pi^*(W_1)) = \mu(W_1)(L.f) > \mu(W)(L.f) = \mu_L(V)
   \end{align*}
 But, $\pi^*(W_1)$ is $\theta$-invariant subbundle of $V$, and hence it contradicts that $(V,\theta)$ is $L$-semistable Higgs bundle. Hence our claim is proved.
\end{proof}
\end{thm}

\section{Stability under pullback}

\begin{thm}\label{thm4.1}
\it Let $L$ be a fixed polarization on a ruled surface $\pi: X\longrightarrow C $. Then, for any stable Higgs bundle $\mathcal{W} = (W,\psi)$ on $C$, the pullback Higgs bundle $\pi^*(\mathcal{W})$ is $L$-stable Higgs bundle on $X$. 
Conversely, if $\mathcal{V} = (V,\theta)$ is a $L$-stable Higgs bundle on $X$ with $c_1(V) = \pi^*(\bf d \it)$ for some divisor \bf d \it on $C$ and $c_2(V)= 0$, then $\mathcal{V} \cong \pi^*(\mathcal{W})$ for some stable Higgs
bundle $\mathcal{W}=(W,\psi)$ on $C$.
\begin {proof}
 If $\pi^*(\mathcal{W})$ is strictly $L$-semistable Higgs bundle, then there is a short exact sequence of torsion-free sheaves
\begin{align}\label{seq4.1.1}
0 \longrightarrow  V_1 \longrightarrow \pi^* (W)\longrightarrow V_2 \otimes I_Z \longrightarrow 0
\end{align}
 where  $V_1$ is a $d\pi(\psi)$-invariant subbundle of $\mathcal{V}$ of rank \it m \rm ,  $V_2$  is vector bundle of rank  \it n \rm on $X$ and \it Z \rm is a closed subscheme of co-dimension 2 in $X$ having $\ell(Z)$ number of points
 in its support, such that
 \begin{align*}
 \mu_L(V_1)=\mu_L(V_2 \otimes \it I_Z) = \mu_L(\pi^*(W)).
 \end{align*}
   Restricting the above exact sequence (\ref{seq4.1.1}) to a generic fiber $f$ such that  Supp(\it Z \rm)$\bigcap f = \emptyset$, we have 
\begin{align}\label{seq4.1.2}
0 \longrightarrow V_1\vert_f \longrightarrow \pi^*(W)\vert_f \longrightarrow V_2\vert_f \longrightarrow 0
\end{align}      
     Since $\pi^*(W)\mid_f  \cong  \mathcal{O}_f^{\oplus r}$, it is slope semistable of degree 0 on $f \cong \mathbb{P}^1$ and hence deg($V_2\mid_f) \geq 0$. Our claim is that deg($V_2\vert_f)> 0$. If not, let deg($V_2\vert_f)=0$. Hence, 
     from the above exact sequence (\ref{seq4.1.2}), for a generic fiber $f$, we get,  deg$(V_1\vert_f)=0$, deg$(V_2\vert_f)=0$ and
     $V_1\vert_f$, $V_2\vert_f$ are semistable on $f\cong \mathbb{P}^1$. Therefore, for a generic fiber $f$, $V_1\vert_f\cong \mathcal{O}^{\oplus m}_f$ and $V_2\vert_f\cong \mathcal{O}^{\oplus n}_f$. By Lemma \ref{lemma3.1},  we have $c_2(V_1)\geq 0$, 
     $c_2(V_2)\geq 0$. We also have from exact sequence (\ref{seq4.1.1})
\begin{align*}    
c_2(\pi^*(W))=\pi^*(c_2(W))= 0 = c_2(V_1)\rm + c_2(V_2 \otimes \it I_Z \rm)= c_2\rm(V_1\rm)+ c_2\rm(V_2\rm) + \it n \rm.\ell(\it Z \rm)
\end{align*}
    which implies c$_2$\rm($V_1$\rm)=c$_2$\rm($V_2$\rm)= $\ell$(\it Z \rm)= 0.  Hence by the Lemma \ref{lemma3.1}, $V_1 \cong \pi^*(W_1$) and $V_2 \cong \pi^*(W_2$) for some vector bundle $W_1$ and $W_2$ on $C$ and \it Z \rm = $\emptyset$ in the exact 
    sequence (\ref{seq4.1.1}). Note that since $V_1$ is $d\pi(\psi)$-invariant, it will imply $W_1$ is $\psi$-invariant. Since in this case, 
\begin{align*}
\mu_L(V_1)= \mu(W_1)(L.f)=\mu_L(\pi^*(W))=\mu(W)(L.f)
\end{align*}
  we have $\mu(W_1) = \mu(W)$ for the $\psi$-invariant subbundle $ W_1 \longrightarrow W$  which contradicts the Higgs stability of $\mathcal{W}$. Thus deg($V_2\vert_f)> 0$ and hence deg($V_1\vert_f)< 0$.    
Now choose a positive integer  $i$  such that $L_i$ := $L+if$ is ample. We then have for all $d\pi(\psi)$-invariant subbundle $0 \longrightarrow M \longrightarrow \pi^*(W) , \mu_{L_i}$($M)< \mu_{L_i}$($\pi^*(W))$. Hence $\pi^*(\mathcal{W})$ 
is ${L_ i}$-stable Higgs bundle and the discriminant of $\pi^*(W)$ being 0, by Remark \ref{remark2}, $\pi^*(\mathcal{W})$ is $L$ -stable Higgs bundle for any polarization $L$ on $X$.
 
 \vspace{2mm}
 
 Conversely, if $\mathcal{V}=( V ,\theta)$ is a $L$-stable Higgs bundle on $X$ with $c_1(V)=\pi^*$(\bf d\rm) for some divisor \bf d \rm on $C$ and $c_2(V)= 0 $, then by Theorem \ref{thm4.1}, $\mathcal{V} \cong \pi^*(\mathcal{W})$ for some 
  semistable Higgs bundle $\mathcal{W}=(W,\psi)$ on $C$. If $\mathcal{W}$ is strictly semistable Higgs bundle, then there is an exact sequence of $\psi$-invariant subbundle of $W$
\begin{align}\label{seq4.1.3}
0 \longrightarrow W_1 \longrightarrow W \longrightarrow W_2 \longrightarrow 0 
\end{align}
such that $\mu(W_1)= \mu(W) =\mu(W_2)$. The exact sequence (\ref{seq4.1.3}) will then pullback to an exact sequence
\begin{align*}
0 \longrightarrow \pi^*(W_1) \longrightarrow V \longrightarrow  \pi^*(W_2) \longrightarrow 0
\end{align*}
such that $\mu_L(\pi^*(W_1)) = \mu_L(V) = \mu_L(\pi^*(W_2))$. Note that $\pi^*(W_1)$ and $\pi^*(W_2)$ are $\theta$-invariant subbundle of $V$. 
Hence, this contradicts that $\mathcal{V}$ is $L$-stable Higgs bundle. Therefore, $\mathcal{W}$  is stable Higgs bundle on $C$ such that $\pi^*(\mathcal{W}) \cong \mathcal{V}$.
\end{proof}
\end{thm}

As a corollary to Theorem \ref{thm4.1}, we have the following result which generalizes the results of Takemoto and Marian Aprodu for rank 2 ordinary vector bundles. ( See in \cite{12} Proposition 3.4, Proposition 3.6 and in  \cite{1} corollary 3 )

\begin{corl}\label{cor4.3}
 \rm Let $L$ be a fixed polarization on a ruled surface $\pi: X\longrightarrow C $. Then, for any stable bundle  $W$ of rank $r$ on $C$, the pullback bundle $\pi^*(W)$ is slope $L$-stable on $X$ . Conversely, if $V$ is a slope $L$-stable vector bundle 
 of rank $r$ on $X$ with $c_1(V)=\pi^*(\bf d \it)$ for some divisor \bf d \rm on $C$ and $c_2(V)= 0$, then $V \cong \pi^*(W)$ for some slope stable vector bundle $W$ on $C$.
\end{corl}

 \subsection{An Example}
\rm Consider the following example. Let $C$ be a smooth complex projective curve of genus $g(C) \geq 2$. Let $E = K^{\frac{1}{2}}_C \oplus K^{-\frac{1}{2}}_C$ , where $K^{\frac{1}{2}}_C$ is a square root of the canonical bundle $K_C$. 
Note,  $K^{2}_C \cong$ Hom$(K^{-\frac{1}{2}}_C , K^{\frac{1}{2}}_C \otimes K_C)$. Then, we obtain a Higgs field $\psi$ on $E$ by setting,
\begin{equation*}
\psi = \begin{pmatrix} 0 & \omega \\ 1 & 0 \end{pmatrix}
\end{equation*}
where $\omega\neq0\in$ Hom $(C,K^{2}_C)$ and 1 is the identity section of trivial bundle Hom$(K^{\frac{1}{2}}_C,K^{-\frac{1}{2}}_C \otimes K_C)$. Now, $(E,\psi)$ is stable Higgs bundle since $K^{\frac{1}{2}}_C$ is not $\psi$-invariant and 
there is no subbundle of positive degree which is preserved by $\psi$. However, $E$ is not slope semistable. Let $\pi:X\longrightarrow C$ be a  ruled surface. In such cases, by Theorem \ref{thm4.1},
 the pullback of these non-trivial stable Higgs bundle on $C$ will prove the existence of  non-trivial stable Higgs bundle on the ruled surface $X$ whose underlying vector bundles are not slope stable.
 
\section{Isomorphism of Moduli spaces}
Let \bf d \rm be a degree d divisor on a curve $C$ and $\pi: X \longrightarrow C$ be a ruled surface on $C$ with a fixed polarization $L$ on $X$. Recall that we denote the moduli space  of S-equivalence classes of Higgs $L$-semistable 
bundles $\mathcal{E}=(V,\theta)$ of rank $r$ on $X$, having $c_1(V)=\pi^*(\bf d\rm)$ and $c_2(V)=0$, by $\mathcal{M}_X^{Higgs}(r,\pi^*(\bf d \rm),0)$. We also denote the moduli space of S-equivalence classes of semistable Higgs bundles of rank $r$ and 
degree $d$ on $C$ by $\mathcal{M}^{Higgs}_C(r,d)$.

We have the following theorem which come up as a corollary to the theorems proved in section(3) and section(4) in this paper.

\begin{thm}\label{thm6}
\rm The moduli spaces $\mathcal{M}_X^{Higgs}(r,\pi^*(\bf d \it),0)$ and $\mathcal{M}^{Higgs}_C(r,d)$ are isomorphic as algebraic varieties.
 \begin{proof}
  Let $M_X^{Higgs}(r,\pi^*(\bf d\rm),0)$ and $M^{Higgs}_C(r,d)$ denote the moduli functors whose corresponding coarse moduli spaces are $\mathcal{M}_X^{Higgs}(r,\pi^*(\bf d \rm),0)$ and $\mathcal{M}^{Higgs}_C(r,d)$ respectively.
  For a given finite-type scheme $T$ over $k$, $M_X^{Higgs}(r,\pi^*(\bf d\rm),0)(T)$ is the set of equivalence classes of flat families of $L$-semistable Higgs Bundles on $X$ of rank $r$ with $c_1(V)=\pi^*(\bf d\rm)$ and $c_2(V)=0$ 
  parametrized by $T$. A family parametrized by $T$ corresponding to $M_X^{Higgs}(r,\pi^*(\bf d\rm),0)$ is a pair ($\mathcal{F},\psi$) where $\mathcal{F}$ is a coherent sheaf on $X\times T$, flat over $T$ and 
  $\psi \in Hom(\mathcal{F},\mathcal{F}\otimes_{\mathcal{O}_{X\times T}}p_1^*(\Omega^1_X))$, where $p_1$ denotes the projection map from $X\times T$ to $X$. Further, for every closed point $t \in T$, we have for the natural
  embedding $ t : X \hookrightarrow X\times T$, the pair $(F_t,\psi_t):=(t^*\mathcal(F),t^*(\psi))$ is a $L$-semistable Higgs bundle of rank $r$ with $c_1(F_t)=\pi^*(\bf d\rm)$ and $c_2(F_t)=0$. Let $\pi_T:=\pi\otimes id_T:X\times T \longrightarrow C\times T$. 
  Then, from Theorem \ref{thm3.5}, we get a flat family ($\mathcal{G},\phi) := ((\pi_T)_*(\mathcal{F}),(\pi_T)_*(\psi))$ parametrized by $T$ corresponding to $M^{Higgs}_C(r,d)$ such that $(\mathcal{F}_t,\psi_t) \cong (\pi^*(\mathcal{G}_t),d\pi(\phi_t))$ 
  with deg($\mathcal{G}_t)= d$ and ($\mathcal{G}_t,\phi_t)$ is a semistable Higgs bundle for every closed point $t\in T$. [Here for every closed point $t\in T$ and for the natural embedding $\tilde{t} :  C \longrightarrow C\times T$, 
  we define $ (\mathcal{G}_t,\phi_t) := (\tilde{t}^*(\mathcal{G}),\tilde{t}^*(\phi))] $. So, we get a natural transformation of functors 
\begin{align*}
 \pi_*: M_X^{Higgs}(r,\pi^*(\bf d\rm),0) \longrightarrow \it M^{Higgs}_C(r,d)
\end{align*}
Similarly, starting from a flat family ($\mathcal{G},\phi$) of semistable Higgs bundles parametrized by $T$ with deg($\mathcal{G}_t) = d$ and rank($\mathcal{G}_t) = r$ 
for every closed point $t$ of $T$, by using Theorem \ref{thm3.3}, we can get a flat family ($\mathcal{F},\psi$) of $L$-semistable Higgs bundles on $X$ parametrized by $T$ such that for 
every closed point $t$ in $T$, $c_1(\mathcal{F}_t)=\pi^*(\bf d )$, $c_2(\mathcal{F}_t)=0$ and $(\mathcal{F}_t,\psi_t) \cong (\pi^*(\mathcal{G}_t),d\pi(\phi_t))$. So, we get a natural transformation of functors
\begin{align*}
 \pi^*: M^{Higgs}_C(r,d) \longrightarrow M_X^{Higgs}(r,\pi^*(\bf d\rm),0)
\end{align*}

By construction, $ \pi_*\circ\pi^* $ and $ \pi^*\circ\pi_* $ are identity transformations on $ \mathcal{M}^{Higgs}_C(r,d) $ and $ \mathcal{M}_X^{Higgs}(r,\pi^*(\bf d \rm),0) $ respectively. 
Hence, the corresponding coarse moduli spaces are also isomorphic as varieties. 
\end{proof}
\end{thm}

\begin{xrem}\label{remark3}
\rm Using Theorem \ref{thm4.1} and a similar kind of argument as in Theorem \ref{thm6}, it can be shown that the moduli space $\mathcal{M}^{stableHiggs}_C(r,d)$ of isomorphic classes of stable Higgs bundles of rank $r$ and 
degree $d$ on $C$ and the moduli space $\mathcal{M}_X^{stableHiggs}(r,\pi^*(\bf d \rm),0)$ of isomorphic classes of Higgs $L$-stable bundles $\mathcal{E}=(V,\theta)$ of rank $r$ on $X$, having $c_1(V)=\pi^*(\bf d\rm)$ for some divisor \bf d \rm of degree 
$d$ on $C$ and $c_2(V)=0$, are isomorphic as algebraic varieties.

\end{xrem}

\subsection*{Acknowledgement}
I would like to thank my advisor Prof. D.S. Nagaraj, IMSc Chennai for his constant guidance at every stage of this work. 
I would also like to thank Dr. Rohith Varma, IMSc Chennai for many useful discussions. This work is supported financially by a fellowship from IMSc,Chennai (HBNI), DAE, Government of India.


\begin{thebibliography}{*****}

\normalsize
\baselineskip=16pt

\bibitem[1]{1} Marian Aprodu, Vasile Br\^{i}nz\v{a}nescu,
\emph{Stable rank-2 vector bundles over ruled surfaces}
C.R.Acad.Sci.Paris,t.325,S\'{e}rie I,p.295-300,(1997)


\bibitem[2]{2} U.Bruzzo, D.Hern\'{a}ndez Ruip\'{e}rez,
\emph{Semistability vs.nefness for (Higgs)vector bundles.}
Differential Geometry and its Applications 24(2006) 403-416.


\bibitem[3]{3} Robert Friedman,
\emph{Algebraic surfaces and Holomorphic Vector Bundles} 
Universitext in Mathematics, 1998. Springer,Berlin.

\bibitem[4]{4} David Gieseker, Jun Li,
\emph{Moduli of Higher Rank Vector Bundles over Surfaces} 
Journal of the American Mathematical Society,Volume 9, Number 1, January 1996.

\bibitem[5]{5} Robin Hartshorne,
\emph{Algebraic Geometry}, Graduate Text in Mathematics, Springer,1977.

\bibitem[6]{6} Daniel Huybrechts and Manfred Lehn,
\emph{The Geometry of Moduli Spaces of Sheaves},
Second Edition , 2010. Cambridge University Press.

\bibitem[7]{7} Adrian Langer,
\emph{Bogomolov's inequality for Higgs sheaves in positive characteristic}
Invent.Math (2015)199:889-920


\bibitem[8]{8} Carlos T. Simpson,
\emph{Constructing Variations of Hodge Structure Using Yang-Mills Theory and Applications to Uniformization}
Journal of American Mathematical society,Vol 1, No 4,(1988).

\bibitem[9]{9} Carlos T. Simpson,
\emph{Moduli of representations of the fundamental group of a smooth projective variety I}
Publ. Math. I.H.E.S. 79 (1994), 47-129.
 
\bibitem[10]{10} Carlos T. Simpson,
\emph{Moduli of representations of the fundamental group of a smooth projective variety II}
Publ. Math. I.H.E.S. 80 (1994), 5-79.

\bibitem[11]{11} Xiaotao Sun,
\emph{Minimal rational curves on moduli spaces of stable bundles},
Math. Ann. 331, 925 - 937 (2005)

\bibitem[12]{12} Fumio Takemoto,
\emph{Stable Vector Bundles on Algebraic Surfaces} 
Nagoya Math. J . Vol. 47 (1972), 29-48


\bibitem[13]{13} Rohith Varma,
\emph{On Higgs Bundles on elliptic surfaces}
Quart. J. Math. 66 (2015), 991–1008.

\end{thebibliography}
\end{document}